\documentclass[12pt, reqno, a4paper]{amsart}
\usepackage[margin=1.2in]{geometry}
\numberwithin{equation}{section}
\usepackage{amssymb,amsfonts,amsthm}
\usepackage[utf8]{inputenc}
\usepackage{listings}
\usepackage{bm}
\usepackage[hyperpageref]{backref}
\usepackage{esint}
\usepackage{bigints}
\usepackage{amsthm}
\usepackage{color}
\usepackage[bookmarks]{hyperref}
\usepackage{siunitx}
\usepackage{hyperref}
\usepackage[T1]{fontenc}
\theoremstyle{plain}
\usepackage{tikz}
\usetikzlibrary{shapes}
\usepackage{caption} 
\usepackage{mathtools}

\allowdisplaybreaks[4] 
\usepackage[T1]{fontenc}
\usepackage{amsthm}

\newtheoremstyle{myremark}{10pt}{10pt}{}{}{\bfseries}{.}{.5em}{}

 \newtheorem{thm}{Theorem}
 
 \newtheorem{lemma}{Lemma}[section]
 
 \theoremstyle{definition}

\usepackage{hyperref}

\allowdisplaybreaks[4]

\begin{document}

\title[Weighted fractional Hardy inequalities]{Weighted fractional Hardy inequalities with singularity on any flat submanifold}

\author{VIVEK SAHU}
\address{ Department of Mathematics and Statistics,
Indian Institute of Technology Kanpur, Kanpur - 208016, Uttar Pradesh, India}

\email{viveksahu20@iitk.ac.in, viiveksahu@gmail.com}



\date{}

\dedicatory{}

\begin{abstract}
 We extend the work of Dyda and Kijaczko by establishing the corresponding weighted fractional Hardy inequalities  with singularities on any flat submanifolds. While they derived weighted fractional Hardy inequalities with singularities at a point and on a half-space, we generalize these results to handle singularities on any flat submanifold of codimension $k$, where $1<k<d$. Furthermore, we also address the critical case $sp=k+\alpha+ \beta$ and establish weighted fractional Hardy inequality with appropriate logarithmic weight function.
\end{abstract}

\maketitle


\section{Introduction}
Fractional Hardy inequality is a fundamental tool in areas such as mathematical physics, the analysis of linear and nonlinear partial differential equations, and harmonic analysis. The inequality is stated as follows (see \cite{frank2008}): let $d \geq 1$ and $s \in (0,1)$. For any function $u \in W^{s,p}(\mathbb{R}^{d})$ when $1 \leq p \leq \frac{d}{s}$, and for any $u \in W^{s,p}(\mathbb{R}^{d} \backslash \{ 0 \})$ when $p > \frac{d}{s}$, the following inequality holds for a optimal constant $C=C(d,s,p)>0$,
\begin{equation}
    \int_{\mathbb{R}^{d}} \frac{|u(x)|^{p}}{|x|^{sp}} dx \leq C \int_{\mathbb{R}^{d}}\int_{\mathbb{R}^{d}} \frac{|u(x)-u(y)|^{p}}{|x-y|^{d+sp}}  dxdy =: [u]^{p}_{W^{s,p}(\mathbb{R}^{d})}.
\end{equation}

\smallskip

In this article, we investigate the Hardy-type inequalities in the context of weighted fractional Sobolev space with singularity on any flat submanifold. We define the weighted fractional Sobolev space as follows:  let $s \in (0,1), ~ p \geq 1$ and $\alpha, \beta \in \mathbb{R}$. For any $x \in \mathbb{R}^{d}$, we denote $x= (x_{k}, x_{d-k})$, where $1 \leq k < d, ~ k \in \mathbb{N}, ~ x_{k} \in \mathbb{R}^{k}$ and $x_{d-k} \in \mathbb{R}^{d-k}$. In the special case $k=d$, this notation refers to $x=(x,0)$. The weighted Gagliardo seminorm for a fixed $1\leq k \leq d$ is then defined as
\begin{equation*}
     [u]_{W^{s,p, \alpha, \beta}(\mathbb{R}^{d})} := \left( \int_{\mathbb{R}^{d}} \int_{\mathbb{R}^{d}}  \frac{|u(x)-u(y)|^{p}}{|x-y|^{d+sp}} |x_{k}|^{\alpha} |y_{k}|^{\beta} dxdy \right)^{\frac{1}{p}},
\end{equation*}
and the corresponding weighted fractional Sobolev space is given by
\begin{equation}
    W^{s,p, \alpha, \beta}(\mathbb{R}^{d}) := \left\{  u \in L^{p}(\mathbb{R}^{d}) : [u]_{W^{s,p, \alpha, \beta}(\mathbb{R}^{d})} < \infty   \right\}.
\end{equation}
For convention, when $\alpha=\beta=0$, we denote $[u]_{W^{s,p, 0, 0}(\mathbb{R}^{d})} := [u]_{W^{s,p}(\mathbb{R}^{d})} $ and refer to the space $W^{s,p, 0, 0}(\mathbb{R}^{d})= W^{s,p}(\mathbb{R}^{d})$ as the usual fractional Sobolev space.

\smallskip

Dyda and Kijaczko in \cite{dyda2022sharp} (see, also \cite{dyda2024}) studied weighted fractional Hardy inequalities corresponding to a singularity at a point (say origin), and established the following result: let $s \in (0,1), ~ p \geq 1$ and $\alpha, ~ \beta, ~ \alpha+ \beta \in (-sp,d)$. For all $u \in C_{c}(\mathbb{R}^{d})$ when $sp+ \alpha+ \beta <d$, and for all $u \in C_{c}(\mathbb{R}^{d} \backslash \{ 0 \})$ when $sp+ \alpha+ \beta >d$, the following inequality holds for a optimal constant $C>0$,
\begin{equation}
    \int_{\mathbb{R}^{d}} \frac{|u(x)|^{p}}{|x|^{sp+\alpha+\beta}} dx \leq C\int_{\mathbb{R}^{d}}\int_{\mathbb{R}^{d}} \frac{|u(x)-u(y)|^{p}}{|x-y|^{d+sp}} |x|^{-\alpha} |y|^{-\beta} dxdy.
\end{equation}
Furthermore, in the same work, they addressed the case $k=1$, which corresponds to a singularity on a hyperplane (say $\{x= (x_{d-1}, x_{d}) \in \mathbb{R}^{d-1} \times \mathbb{R} : x_{d}=0 \}$) and they proved the following result: let $\alpha, ~ \beta, ~ \alpha+ \beta \in (-1,sp)$ and $1+ \alpha + \beta \neq sp$. Then for all $u \in C_{c}(\mathbb{R}^{d}_{+})$ with a optimal constant $C>0$,
\begin{equation}
    \int_{\mathbb{R}^{d}_{+}} \frac{|u(x)|^{p}}{x^{sp-\alpha-\beta}_{d}} dx \leq C \int_{\mathbb{R}^{d}_{+}}\int_{\mathbb{R}^{d}_{+}} \frac{|u(x)-u(y)|^{p}}{|x-y|^{d+sp}} x^{\alpha}_{d} y^{\beta}_{d} dxdy.
\end{equation}
The critical case ~$1+ \alpha + \beta =sp$ was addressed in our recent work \cite[Theorem 6]{adimurthi2024fractional}, where we identified the appropriate logarithmic weight function to establish the weighted fractional Hardy inequality with singularity on a hyperplane. For recent advancements in fractional Hardy-type inequalities, we refer to \cite{adimurthi2024boundaryfractionalhardysinequality, brasco2024, brasco2018, csato, Cabre2022, Dyda2004, dyda2022, dyda2022sharp, dyda2018, frank2008, Frank2010, leonibook, squassina2018, firoz}.
\smallskip

Our aim in this article is to extend the above results to address singularities on any flat submanifold of codimension $k$, where $1<k<d, ~ k \in \mathbb{N}$. In particular, we consider the flat submanifold  
\begin{equation}\label{k defn}
   K:= \{x= (x_{k}, x_{d-k}) \in \mathbb{R}^{k} \times \mathbb{R}^{d-k} : x_{k} =0 \},
\end{equation} with $1<k<d$ and establish weighted fractional Hardy inequalities with singularity on $K$. The following theorem establishes our main results for the case $sp \neq k+ \alpha + \beta$, where $\alpha, \beta \in \mathbb{R}$ such that $sp-\alpha-\beta>0$. In particular we prove the following theorem:

\begin{thm}\label{theorem1}
Let $d \geq 3, ~ p>1, ~ s \in (0,1)$ and $1<k<d, ~ k \in \mathbb{N}$. Assume $\alpha, \beta \in \mathbb{R}$ are such that $sp-\alpha-\beta>0$ and $sp \neq k+ \alpha+ \beta$. Let $K$ be a flat submanifold defined in \eqref{k defn}. Then, there exists a constant $C=C(d,p,s,k,\alpha, \beta)>0$ such that
     \begin{equation}
        \int_{\mathbb{R}^{d}} \frac{|u(x)|^{p}}{|x_{k}|^{sp-\alpha-\beta}} dx \leq C \int_{\mathbb{R}^{d}} \int_{\mathbb{R}^{d}} \frac{|u(x)-u(y)|^{p}}{|x-y|^{d+sp}} |x_{k}|^{\alpha} |y_{k}|^{\beta} dxdy, \hspace{3mm} \forall \ u \in C^{1}_{c}(\mathbb{R}^{d} \backslash K).
    \end{equation}
\end{thm}

\smallskip

The case $sp=k+ \alpha+ \beta$, where $1<k<d$ and $\alpha, \beta \in \mathbb{R}$ is considered critical. In our recent work in \cite{adimurthiSubmanifold}, we addressed this critical case with $\alpha=\beta=0$ and established the fractional Hardy inequality with singularities on smooth compact sets of codimension $k$, where $1<k<d$ in a bounded domain, incorporating an optimal logarithmic weight function. Motivated by these results, we now derive the corresponding weighted fractional Hardy inequality with a logarithmic weight function for the critical case $sp = k + \alpha + \beta$. Let $B^{k}(0, R)$ denote a ball of radius $R$ centered at $0$ in $\mathbb{R}^{k}$. With this framework, we present the following theorem for the critical case:

\begin{thm}\label{theorem2}
    Let $d \geq 3, ~ p>1, ~ s \in (0,1)$ and $1<k<d, ~ k \in \mathbb{N}$. Assume $\alpha, \beta \in \mathbb{R}$ are such that $sp = k + \alpha + \beta$. Let $K$ be a flat submanifold defined in \eqref{k defn}.  Then, there exists a constant $C=C(d,s,p, \alpha, \beta)>0$ such that for any  $u \in C^{1}_{c}(\mathbb{R}^{d} \backslash K)$ satisfying $supp(u) \subset B^{k}(0, R) \times \mathbb{R}^{d-k}$ for some $R>0$,
   \begin{equation}
        \bigintsss_{\mathbb{R}^{d}} \frac{|u(x)|^{p}}{|x_{k}|^{k} \ln^{p} \left( \frac{4R}{|x_{k}|}  \right) }  dx \leq C \int_{\mathbb{R}^{d}} \int_{\mathbb{R}^{d}} \frac{|u(x)-u(y)|^{p}}{|x-y|^{d+sp}} |x_{k}|^{\alpha} |y_{k}|^{\beta} dxdy.
   \end{equation}
\end{thm}

\smallskip

The motivation for proving the weighted fractional Hardy inequalities with singularity on flat submanifold obtained from our recent work \cite{adimurthiSubmanifold}, where we derived the full range of fractional Hardy inequalities with singularities on smooth compact sets of codimension $k$, where $1<k<d$.

\subsection*{Applications:} Fractional Hardy inequalities with singularities on flat submanifold have applications in studying fractional Laplacian weighted problems. For example, consider the following fractional $p$-Laplacian weighted problem:
\begin{equation}\label{problem}
   (- \Delta_{p})^{s}u(x) = \lambda \frac{|u(x)|^{p-2}u(x)}{|x_{k}|^{sp}}, \hspace{3mm} x= (x_{k}, x_{d-k}) \in \mathbb{R}^{d} \backslash K,
\end{equation}
where $K$ is defined in \eqref{k defn}, $1<k<d$ and $\lambda>0$. Assuming $ sp \neq k$ and applying Theorem \ref{theorem1} with $\alpha=\beta=0$, it would be interesting to study the above problem.

\smallskip

The article is organized in the following way: In Section ~\ref{preliminaries}, we introduce the preliminary lemmas and notations necessary for proving our main results. Section ~\ref{proof of main result} presents the proof of our main results.

\section{Notations and Preliminaries}\label{preliminaries} 
In this section, we present the notations and preliminary results that will be used throughout the article. The following notations will be used: for a measurable set ~$\Omega \subset \mathbb{R}^d, ~ (u)_{\Omega}$ will denote the average of the function ~$u$ over ~$\Omega$, i.e., 
\begin{equation*}
    (u)_{\Omega} :=  \frac{1}{|\Omega|} \int_{\Omega} u(x)  dx = \fint_{\Omega} u(x)dx. 
\end{equation*}
Here, ~$|\Omega|$ represents the Lebesgue measure of ~$\Omega$. ~$s$ will always be understood to be in ~$(0,1)$. Any point ~$x \in \mathbb{R}^d$ is represented as ~$x=(x_{k},x_{d-k})$ where ~$x_{k} \in \mathbb{R}^{k}$ and ~$x_{d-k} \in \mathbb{R}^{d-k}$. ~$C>0$ will denote a generic constant that may change from line to line.

\bigskip

{\bf Fractional Poincar\'e\ Inequality} ~\cite[Theorem 3.9]{edmunds2022}: Let ~$p \geq 1$ and ~$\Omega \subset \mathbb{R}^d$ be a bounded open set. Then there exists a constant ~$C = C(s,p, \Omega)>0$ such that
\begin{equation}\label{poincare}
    \int_{\Omega} |u(x)-(u)_{\Omega}|^{p}  dx \leq C [u]^{p}_{W^{s, p}(\Omega)}, \hspace{3mm} \forall \ u \in W^{s,p}(\Omega) .
\end{equation} 

\smallskip

The next lemma provides a fractional Poincar\'e\ inequality on a bounded open set with a scaling parameter $\lambda>0$. A similar type of the following lemma for bounded Lipschitz domain is available in ~\cite[Lemma $2.1$]{adimurthiSubmanifold}. This lemma is crucial for proving our main results.

\begin{lemma}\label{sobolev} 
   Let ~$\Omega$ be a bounded open set in ~$\mathbb{R}^{d}, ~ d \geq 1, ~ p \geq 1$ and ~$s \in (0,1)$. Define ~$\Omega_{\lambda}: = \left\{ \lambda x : ~ x \in \Omega \right\}$ for ~$\lambda>0$, there exists a constant ~$C=C(d,p,s, \Omega)>0$ such that
    \begin{equation}
          \fint_{\Omega_{\lambda}} |u(x)-(u)_{\Omega_{\lambda}}|^{p} dx  \leq C  \lambda^{sp-d} [u]^{p}_{W^{s, p}(\Omega_{\lambda})} , \hspace{3mm} \forall \ u \in W^{s, p}(\Omega_{\lambda})  .
    \end{equation}
\end{lemma} 
\begin{proof}
  Let ~$\Omega \subset \mathbb{R}^d$ be a bounded open set. From \eqref{poincare}, we have
\begin{equation}\label{ineqlemma}
     \fint_{\Omega} |u(x)-(u)_{\Omega}|^{p} \ dx   \leq C  [u]^{p}_{W^{s, p}(\Omega)}  . 
    \end{equation}
    Let us apply the above inequality to ~$u(\lambda x)$ instead of ~$u(x)$. This gives
    \begin{equation*}
        \fint_{\Omega}  \Big|u(\lambda x)-\fint_{\Omega} u(\lambda x) dx \Big|^{p } \ dx  \leq C \int_{\Omega} \int_{\Omega} \frac{|u(\lambda x) - u(\lambda y)|^{p}}{|x-y|^{d+sp}} dxdy    . 
    \end{equation*}
    Using the fact
    \begin{equation*}
        \fint_{\Omega} u(\lambda x) dx = \fint_{\Omega_{\lambda}} u(x) dx,
    \end{equation*}
    we have
    \begin{equation*}
         \fint_{\Omega} |u(\lambda x)-(u)_{\Omega_{\lambda}}|^{p } \ dx  \leq C  \int_{\Omega} \int_{\Omega} \frac{|u(\lambda x) - u(\lambda y)|^{p}}{|x-y|^{d+sp}} dxdy   . 
    \end{equation*}
    By changing the variable ~$X=\lambda x$ and ~$Y= \lambda y$, we obtain 
    \begin{equation*}
   \fint_{\Omega_{\lambda}} |u(x)-(u)_{\Omega_{\lambda}}|^{p } dx   \leq C \lambda^{sp-d} [u]^{p}_{W^{s, p}(\Omega_{\lambda})}  . 
    \end{equation*}
    This finishes the proof of the lemma.
\end{proof}

\smallskip

The next lemma establishes a relationship between the averages of $u$ over two disjoint sets. This result is available in our recent work ~\cite[Lemma $2.2$]{adimurthiSubmanifold}.

\begin{lemma}\label{avg}
    Let ~$E$ and ~$F$ be disjoint set in ~$\mathbb{R}^d$. Then for any ~$p \geq 1$ one has for some constant ~$C >0$ such that
    \begin{equation}
        |(u)_{E} - (u)_{F}|^{p} \leq C \frac{|E \cup F|}{\min \{ |E|, |F| \} }  \fint_{E \cup F} |u(x)-(u)_{E \cup F}|^{p}dx  . 
    \end{equation}
\end{lemma}
\begin{proof}
See \cite[Lemma $2.2$]{adimurthiSubmanifold} for the proof.
\end{proof}
\smallskip

The following lemma presents an inequality that is instrumental in establishing our results. A similar version of this lemma can also be found in our recent work \cite{adimurthi2024fractional, adimurthiSubmanifold}.

\begin{lemma}\label{estimate}
    Let ~$p >1$ and $c>1$. Then for all ~$a, ~ b \in \mathbb{R}$, we have 
    \begin{equation}
        (|a| + |b|)^{p} \leq c|a|^{p} + (1-c^{\frac{-1}{p -1}})^{1-p} |b|^{p} .
    \end{equation}
\end{lemma}
\begin{proof}
    See ~\cite[Lemma $2.5$]{adimurthiSubmanifold} for the proof.
\end{proof}


\section{Proof of the main results}\label{proof of main result}

In this section, we establish our main results. We begin by considering a domain of the form $ \mathcal{D}:= B^{k}(0, 2^{n_{0}+1}) \times (-2^{n_{1}},2^{n_{1}})^{d-k}$, where $n_{0}, n_{1} \in \mathbb{Z}$ and $2^{n_{1}} \geq 2^{n_{0}+1}$ and obtaining that the constant $C>0$ does not depends on $n_{0}$ and $n_{1}$. Here, $B^{k}(0, 2^{n_{0}+1})$ denotes a ball centered at $0$ with radius $ 2^{n_{0}+1}$ in $\mathbb{R}^{k}$. This choice is motivated by the fact that for any function $u \in C^{1}_{c}(\mathbb{R}^{d} \backslash K)$, where  $K = \{ x=(x_{k}, x_{d-k}) \in \mathbb{R}^{k} \times \mathbb{R}^{d-k} : x_{k} =0 \}$, we can always find suitable values of $n_{0}, n_{1} \in \mathbb{Z}$ such that $supp(u) \subset \mathcal{D}$. 

\smallskip

For ~$1 < k < d, ~ k \in \mathbb{N}$. Choose $n_{0}, ~ n_{1} \in \mathbb{Z}$ such that $2^{n_{1}} \geq 2^{n_{0}+1}$. Let ~$x = (x_{k}, x_{d-k}) \in \mathbb{R}^{k} \times \mathbb{R}^{d-k}$, where ~$x_{k} \in \mathbb{R}^k$ and ~$x_{d-k} \in \mathbb{R}^{d-k}$. Define  
\begin{equation*}
    \mathcal{D} = \left\{ x= (x_{k}, x_{d-k}) :  |x_{k}| < 2^{n_{0}+1} \ \text{and} \ x_{d-k} \in (-2^{n_{1}}, 2^{n_{1}})^{d-k}   \right\},
\end{equation*}
For each ~$\ell \leq n_{0}$, we define
\begin{equation*}
    \mathcal{B}_{\ell} = \{ x=(x_{k},x_{d-k}) \in \mathcal{D} :  2^{ \ell}  \leq |x_{k}| < 2^{\ell +1}  \},
\end{equation*}
Theferore, we have
\begin{equation*}
    \mathcal{D} = \bigcup_{\ell= - \infty}^{n_{0}} \mathcal{B}_{\ell}.
\end{equation*}
Again, we further divide ~$\mathcal{B}_{\ell}$ into a disjoint union of identical sets, denoted as ~$\mathcal{B}^{i}_{\ell}$, such that if ~$x = (x_{k}, x_{d-k}) \in \mathcal{B}^{i}_{\ell}$, then ~$x_{d-k} \in C^{i}_{\ell}$, where ~$C^{i}_{\ell}$  is a cube of side length ~$2^{ \ell}$ in $\mathbb{R}^{d-k}$. Then, we have
\begin{equation*}
      \mathcal{B}_{\ell} = \bigcup_{i = 1}^{\sigma_{\ell}} \mathcal{B}^{i}_{\ell}  ,
\end{equation*}
where ~$\sigma_{\ell} = 2^{(- \ell+1)(d-k)} 2^{n_{1}(d-k)}$. The Lebesgue measure of ~$\mathcal{B}^{i}_{\ell}$ is then given by
\begin{equation*}
    |\mathcal{B}^{i}_{\ell}| = \left( \frac{|\mathbb{S}^{k-1}| (2^{k}-1)  2^{\ell k}}{k} \right) \times 2^{\ell (d-k)} =  \frac{|\mathbb{S}^{k-1}| (2^{k}-1)  2^{\ell d}}{k} .
\end{equation*}
Here, $\mathbb{S}^{k-1}$ represents the boundary of an open ball of radius ~$1$ with center ~$0$ in ~$\mathbb{R}^{k}$.
\begin{figure}[!ht] 
\begin{center}
\begin{tikzpicture}
\node at (-2,.5) {$2^{\ell}\leq |x_{k}| < 2^{\ell +1}$};
    
\draw (0,1) -- (3,1);
\draw (0,-1) -- (3,-1);
\draw[dashed] (0,.5) -- (3,.5);
\draw[dashed] (0,-.5) -- (3,-.5);

\draw (0,0) ellipse [y radius=1, x radius=0.5];
\fill[blue!20, opacity=0.5] (0,0) ellipse [y radius=1, x radius=0.5];
\fill[white, opacity=1] (0,0) ellipse [y radius=0.5, x radius=0.25];
\draw (0,0) ellipse [y radius=.5, x radius=0.25];

\draw (3,0) ellipse [y radius=1, x radius=0.5];
\fill[blue!20, opacity=0.5] (3,0) ellipse [y radius=1, x radius=0.5];
\fill[white, opacity=1] (3,0) ellipse [y radius=0.5, x radius=0.25];
\draw (3,0) ellipse [y radius=.5, x radius=0.25];
\node at (1.5,1.5) {$C^{i}_{\ell}$};
\end{tikzpicture}
\end{center}
\caption{$\mathcal{B}^{i}_{\ell}= \{ x=(x_{k},x_{d-k}) \in \mathcal{B}_{\ell} : x_{d-k} \in C^{i}_{\ell} \}   $.}
\label{fig:myfigure3}
\end{figure}
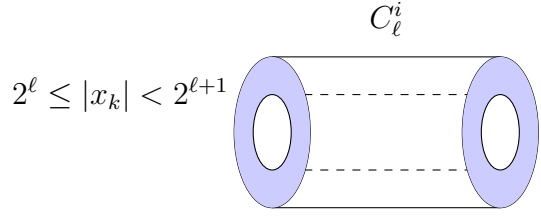

\smallskip

The following lemma provides an inequality that holds true for each ~$\mathcal{B}^{i}_{\ell}$. This lemma is helpful in establishing our main results.

\begin{lemma}\label{sumineqlemma}
For any ~$\mathcal{B}^{i}_{\ell}$ and $\alpha,  \beta \in \mathbb{R}$ such that $sp- \alpha- \beta>0$, there exists a constant ~$C=C(d,p,s, \alpha, \beta)>0$ such that the following inequality holds
   \begin{equation}
 \int_{\mathcal{B}^{i}_{\ell}} \frac{|u(x)|^{p}}{|x_{k}|^{sp- \alpha - \beta }}  dx \leq C [u]^{p}_{W^{s, p, \alpha, \beta}(\mathcal{B}^{i}_{\ell})} + C 2^{\ell (d-sp+ \alpha+ \beta)} |(u)_{\mathcal{B}^{i}_{\ell}}|^{p}  .
\end{equation} 
\end{lemma}
\begin{proof}
    Fix any  ~$\mathcal{B}^{i}_{\ell}$. By applying Lemma \ref{sobolev} with the domain ~$\Omega = \{ (x_{k}, x_{d-k}) : 1 < |x_{k}| < 2  \ \text{and} \ x_{d-k} \in (1,2)^{d-k} \}$, ~$ \lambda = 2^{\ell}$ and using translation invariance, we have
\begin{equation}\label{eqn1}
    \fint_{\mathcal{B}^{i}_{\ell}} |u(x)-(u)_{\mathcal{B}^{i}_{\ell}}|^{p}  dx \leq C 2^{\ell \left(sp-d \right) }[u]^{p}_{W^{s, p}(\mathcal{B}^{i}_{\ell})}  ,
\end{equation}
where ~$C= C(d,p,s)$ is a positive constant. For ~$x = (x_{k},x_{d-k}) \in \mathcal{B}^{i}_{\ell}$, we have ~$ \frac{1}{|x_{k}|} \leq \frac{1}{2^{\ell}}$. Therefore, we have
\begin{eqnarray*}
    \int_{\mathcal{B}^{i}_{\ell}} \frac{|u(x)|^{p}}{|x_{k}|^{sp- \alpha- \beta}}  dx &\leq& \frac{1}{2^{\ell (sp- \alpha- \beta)}} \int_{\mathcal{B}^{i}_{\ell}} |u(x)-(u)_{\mathcal{B}^{i}_{\ell}} + (u)_{\mathcal{B}^{i}_{\ell}}|^{p}  dx \\
    &\leq& \frac{2^{p-1}}{2^{\ell (sp- \alpha- \beta)}} \int_{\mathcal{B}^{i}_{\ell}} |u(x)-(u)_{\mathcal{B}^{i}_{\ell}}|^{p}  dx + \frac{2^{p-1}}{2^{\ell (sp- \alpha- \beta)}} \int_{\mathcal{B}^{i}_{\ell}} |(u)_{\mathcal{B}^{i}_{\ell}}|^{p}dx,
    \end{eqnarray*}
  Now, applying inequality \eqref{eqn1} together with the fact that  $2^{\ell (\alpha+ \beta)} \leq C |x_{k}|^{\alpha}|y_{k}|^{\beta}$ for some positive constant $C$ depends on $\alpha$ and $\beta$, and for all $(x_{k}, x_{d-k}), (y_{k}, y_{d-k}) \in \mathcal{B}^{i}_{\ell}$, we derive
    \begin{eqnarray*}
    \int_{\mathcal{B}^{i}_{\ell}} \frac{|u(x)|^{p}}{|x_{k}|^{sp-\alpha-\beta}}  dx &\leq& C \frac{|\mathcal{B}^{i}_{\ell}|}{2^{\ell (sp-\alpha-\beta)}} \fint_{\mathcal{B}^{i}_{\ell}} |u(x)-(u)_{\mathcal{B}^{i}_{\ell}}|^{p}  dx  + C \frac{|\mathcal{B}^{i}_{\ell}|}{2^{\ell (sp-\alpha-\beta)}} |(u)_{\mathcal{B}^{i}_{\ell}}|^{p} \\
   & \leq & C 2^{\ell (\alpha+ \beta)} [u]^{p}_{W^{s, p}(\mathcal{B}^{i}_{\ell})}  + C 2^{\ell (d-sp+ \alpha+ \beta )} |(u)_{\mathcal{B}^{i}_{\ell}}|^{p} \\ &\leq&   C[u]^{p}_{W^{s,p, \alpha, \beta}(\mathcal{B}^{i}_{\ell})} + C 2^{\ell (d-sp+ \alpha+ \beta )} |(u)_{\mathcal{B}^{i}_{\ell}}|^{p},
\end{eqnarray*}   
where ~$C= C(d,p,s, \alpha, \beta)$ is a positive constant.
\end{proof}

\smallskip

The following lemma establishes a connection between the averages of two disjoint sets ~$\mathcal{B}^{j}_{\ell+1}$ and ~$\mathcal{B}^{i}_{\ell}$, where ~$\mathcal{B}^{j}_{\ell+1}$ is chosen in such a way that ~$C^{i}_{\ell} \subset C^{j}_{\ell+1}$, i.e., 
\begin{equation}\label{codn on Al and Al+1}
    \mathcal{B}^{i}_{\ell} \subset \{ (x_{k}, x_{d-k}) \in \mathcal{B}_{\ell} : x_{d-k} \in C^{j}_{\ell +1}   \}.
\end{equation}
\begin{figure}[!ht] 
\begin{center}
\begin{tikzpicture}

\draw (0,1.5) -- (6,1.5);
\draw (0,-1.5) -- (6,-1.5);

\draw (2,1) -- (4,1);
\draw (2,-1) -- (4,-1);

\draw (0,0) ellipse [y radius=1.5, x radius=1];
\fill[blue!20, opacity=0.5] (0,0) ellipse [y radius=1.5, x radius=1];
\fill[white, opacity=1] (0,0) ellipse [y radius=1, x radius=0.5];
\draw (0,0) ellipse [y radius=.5, x radius=0.25];
\draw (2,0) ellipse [y radius=1, x radius=0.5];
\fill[blue!20, opacity=0.5] (2,0) ellipse [y radius=1, x radius=.5];
\fill[green!20, opacity=0.5] (2,0) ellipse [y radius=1, x radius=.5];
\fill[white, opacity=1] (2,0) ellipse [y radius=.5, x radius=0.25];
\draw (0,0) ellipse [y radius=1, x radius=0.5];
\draw (2,0) ellipse [y radius=.5, x radius=0.25];

\draw (6,0) ellipse [y radius=1.5, x radius=1];    
\fill[blue!20, opacity=0.5] (6,0) ellipse [y radius=1.5, x radius=1];
\fill[white, opacity=1] (6,0) ellipse [y radius=1, x radius=0.5];
\draw (6,0) ellipse [y radius=.5, x radius=0.25];
\draw (4,0) ellipse [y radius=1, x radius=0.5];
\fill[blue!20, opacity=0.5] (4,0) ellipse [y radius=1, x radius=.5];
\fill[green!20, opacity=0.5] (4,0) ellipse [y radius=1, x radius=.5];
\fill[white, opacity=1] (4,0) ellipse [y radius=.5, x radius=0.25];
\draw (6,0) ellipse [y radius=1, x radius=0.5];
\draw (4,0) ellipse [y radius=.5, x radius=0.25];

\draw[dashed] (0,1) -- (6,1);
\draw[dashed] (0,-1) -- (6,-1);
\draw[dashed] (0,-.5) -- (6,-.5); 
\draw[dashed] (0,.5) -- (6,.5);

\node at (2,.7) { $\mathcal{B}^{i}_{\ell}$ };
\node at (0,1.22) {$\mathcal{B}^{j}_{\ell +1}$};
\node at (2.4,2) {$C^{j}_{\ell+1}$};
\node at (3,1.25) {$C^{i}_{\ell}$};
\end{tikzpicture}
\end{center}
\caption{$\mathcal{B}^{i}_{\ell}$ and $\mathcal{B}^{j}_{\ell +1}$ .}
\label{fig:myfigure4}
\end{figure}
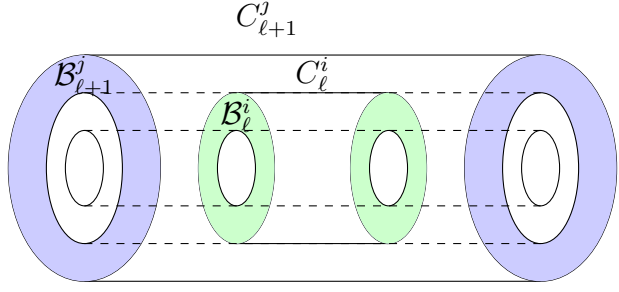
 
\noindent Also, there are $2^{d-k}$ such $\mathcal{B}^{i}_{\ell}$'s which satisfies the above relation. This lemma plays a crucial role in establishing our main results.

\begin{lemma}\label{est2}
    Let $\mathcal{B}^{i}_{\ell}$ and $\mathcal{B}^{j}_{\ell+1}$ be the set such that it satisfies ~\eqref{codn on Al and Al+1}. Then
    \begin{equation}
        |(u)_{\mathcal{B}^{i}_{\ell}} - (u)_{\mathcal{B}^{j}_{\ell+1}} |^{p} \leq C 2^{\ell \left(sp-d \right)} [u]^{p}_{W^{s, p}(\mathcal{B}^{i}_{\ell} \cup \mathcal{B}^{j}_{\ell+1})},
    \end{equation}
    where ~$C$ is a positive constant independent of ~$\ell$.
\end{lemma}
\begin{proof}
Let us consider ~$|(u)_{\mathcal{B}^{i}_{\ell}} - (u)_{\mathcal{B}^{j}_{\ell+1}} |^{p}$  and applying Lemma ~\ref{avg} with the sets ~$E = \mathcal{B}^{i}_{\ell}$ and ~$F = \mathcal{B}^{j}_{\ell+1}$, we obtain for some constant ~$C$ (independent of ~$\ell$), 
\begin{equation}\label{compineq1}
|(u)_{\mathcal{B}^{i}_{\ell}} - (u)_{\mathcal{B}^{j}_{\ell+1}} |^{p} \leq C \fint_{\mathcal{B}^{i}_{\ell} \cup \mathcal{B}^{j}_{\ell+1} }  |u(x) - (u)_{\mathcal{B}^{i}_{\ell} \cup \mathcal{B}^{j}_{\ell+1}}|^{p} dx.
\end{equation}
Choose an open set ~$\Omega$ such that ~$\Omega_{\lambda}$ is a translation of ~$ \mathcal{B}^{i}_{\ell} \cup \mathcal{B}^{j}_{\ell+1}$ with scaling parameter ~$\lambda=2^{\ell+1}$. Applying Lemma ~\ref{sobolev} with this ~$\Omega$ and ~$\lambda=2^{\ell+1}$, and using translation invariance, we obtain
\begin{equation}\label{compineq2}
     \fint_{\mathcal{B}^{i}_{\ell} \cup \mathcal{B}^{j}_{\ell+1} }  |u(x) - (u)_{\mathcal{B}^{i}_{\ell} \cup \mathcal{B}^{j}_{\ell+1}}|^{p} dx \leq C 2^{\ell \left(sp-d \right)} [u]^{p}_{W^{s, p}(\mathcal{B}^{i}_{\ell} \cup \mathcal{B}^{j}_{\ell+1})}.
\end{equation}
By combining ~\eqref{compineq1} and ~\eqref{compineq2}, we establish the lemma.
\end{proof}

\smallskip

The next lemma calculate a inequality which is useful in establishing our main results.

\begin{lemma}\label{lemma 3}
    For any $u \in W^{s,p, \alpha, \beta}(\mathbb{R})$, the following inequality holds
    \begin{equation}\label{se1}
     \sum_{\ell= m}^{n_{0}} [u]^{p}_{W^{s,p, \alpha, \beta}(\mathcal{B}_{\ell} \cup \mathcal{B}_{\ell+1})} \leq 2 [u]^{p}_{W^{s,p, \alpha, \beta}(\mathbb{R}^{d})}.
\end{equation}
\end{lemma}
\begin{proof}
   Consider two families of sets:
\begin{equation*}
  \mathcal{E}:=  \left\{ \mathcal{B}_{\ell} \cup \mathcal{B}_{\ell+1} : |\ell| \ \text{is even and} \ \ell \leq n_{0}  \right\}
\end{equation*}
and
\begin{equation*}
  \mathcal{O}:=  \left\{ \mathcal{B}_{\ell} \cup \mathcal{B}_{\ell+1} : |\ell| \ \text{is odd and} \ \ell \leq n_{0}  \right\}.
\end{equation*}
Then ~$\mathcal{E}$ and ~$\mathcal{O}$ are collection of mutually disjoint sets respectively. Define
\begin{equation*}
    \mathcal{F}_{e} := \bigcup_{\substack{\ell =m \\ |\ell| \ \text{is even}}}^{n_{0}} \mathcal{B}_{\ell} \cup \mathcal{B}_{\ell+1} \hspace{3mm} \text{and} \hspace{3mm} \mathcal{F}_{o} := \bigcup_{\substack{\ell =m \\ |\ell| \ \text{is odd}}}^{n_{0}} \mathcal{B}_{\ell} \cup \mathcal{B}_{\ell+1}.
\end{equation*} 
Therefore, we have
\begin{eqnarray}\label{sum Ak A(k+1)}
    \sum_{\ell= m}^{n_{0}} [u]^{p}_{W^{s,p, \alpha, \beta}(\mathcal{B}_{\ell} \cup \mathcal{B}_{\ell+1})} &=& \sum_{\substack{\ell =m \\ |\ell| \ \text{is even}}}^{n_{0}} [u]^{p}_{W^{s,p, \alpha, \beta}(\mathcal{B}_{\ell} \cup \mathcal{B}_{\ell+1})} + \sum_{\substack{\ell =m \\ |\ell| \ \text{is odd}}}^{n_{0}} [u]^{p}_{W^{s,p, \alpha, \beta}(\mathcal{B}_{\ell} \cup \mathcal{B}_{\ell+1})} \nonumber \\ &\leq& [u]^{p}_{W^{s,p, \alpha, \beta}(\mathcal{F}_{e})} + [u]^{p}_{W^{s,p, \alpha, \beta}(\mathcal{F}_{o})} \leq 2 [u]^{p}_{W^{s,p, \alpha, \beta}(\mathbb{R}^{d})}.
\end{eqnarray}
This establishes the desired inequality. 
\end{proof}

\smallskip

We will now prove Theorem \ref{theorem1} for the case when $sp \neq k+ \alpha+ \beta$. We first address the case $sp<k+ \alpha+ \beta$ establishing Theorem \ref{theorem1}, followed by the case $sp> k+ \alpha+ \beta$ which will be treated seperately. The next subsection focuses on proving Theorem \ref{theorem1} for the case $sp<k+ \alpha+ \beta$. We assume $K = \{x= (x_{k}, x_{d-k}) \in \mathbb{R}^{k} \times \mathbb{R}^{d-k} : x_{k}=0 \}$.

\subsection{Proof of Theorem \ref{theorem1} when \texorpdfstring{$sp<k+\alpha+\beta$}{sp<k}}
Let ~$u \in C^{1}_{c}(\mathbb{R}^{d} \backslash K)$ such that ~$supp(u) \subset B^{k}(0, 2^{n_{0}+1}) \times (-2^{n_{1}},2^{n_{1}})^{d-k} = \mathcal{D}$ for some ~$n_{0}, n_{1} \in \mathbb{Z}$ satisfying ~$2^{n_{1}} \geq 2^{n_{0}+1}$. From Lemma ~\ref{sumineqlemma}, we have
\begin{equation*}
     \int_{\mathcal{B}^{i}_{\ell}} \frac{|u(x)|^{p}}{|x_{k}|^{sp-\alpha-\beta}}  dx \leq C [u]^{p}_{W^{s, p, \alpha, \beta}(\mathcal{B}^{i}_{\ell})} + C 2^{\ell (d-sp+ \alpha+ \beta)} |(u)_{\mathcal{B}^{i}_{\ell}}|^{p}  .
\end{equation*}
By summing the above inequality from ~$i=1$ to ~$\sigma_{\ell}$, we get
\begin{eqnarray*}
    \int_{\mathcal{B}_{\ell}} \frac{|u(x)|^{p}}{|x_{k}|^{sp-\alpha-\beta}}  dx &\leq& C   \sum_{i=1}^{\sigma_{\ell}} [u]^{p}_{W^{s, p, \alpha, \beta}(\mathcal{B}^{i}_{\ell})} + C 2^{\ell (d-sp+ \alpha+ \beta)}  \sum_{i=1}^{\sigma_{\ell}} |(u)_{\mathcal{B}^{i}_{\ell}}|^{p}  \\
    &\leq& C [u]^{p}_{W^{s, p, \alpha, \beta}(\mathcal{B}_{\ell})} +  C 2^{\ell (d-sp+ \alpha+ \beta)}  \sum_{i=1}^{\sigma_{\ell}} |(u)_{\mathcal{B}^{i}_{\ell}}|^{p}.
\end{eqnarray*}
Again, summing the above inequality from ~$\ell=m \in \mathbb{Z}^{-}$ to ~$n_{0}$, we obtain
\begin{equation}\label{eqnn1}
\sum_{\ell=m}^{n_{0}} \int_{\mathcal{B}_{\ell}} \frac{|u(x)|^{p}}{|x_{k}|^{sp-\alpha-\beta}}  dx \leq  C \sum_{\ell=m}^{n_{0}} [u]^{p}_{W^{s, p, \alpha, \beta}(\mathcal{B}_{\ell})} + C \sum_{\ell=m}^{n_{0}} 2^{\ell (d-sp+ \alpha+ \beta)} \sum_{i=1}^{\sigma_{\ell}} |(u)_{\mathcal{B}^{i}_{\ell}}|^{p}  .
\end{equation} 
Let ~$\mathcal{B}^{j}_{\ell+1}$ be such that it satisfies ~\eqref{codn on Al and Al+1}. Using triangle inequality, we have 
\begin{equation*}
    |(u)_{\mathcal{B}^{i}_{\ell}}|^p \leq  \left( |(u)_{\mathcal{B}^{j}_{\ell+1}}| + |(u)_{\mathcal{B}^{i}_{\ell}} - (u)_{\mathcal{B}^{j}_{\ell+1}}| \right)^p.
\end{equation*}
We know that $sp<k+ \alpha+\beta$, applying Lemma ~\ref{estimate} with ~$c :=c_{1}2^{k - sp+ \alpha+ \beta}>1$ where ~$c_{1} = \frac{2}{1+ 2^{k-sp+ \alpha+ \beta}} <1$ together with Lemma ~\ref{est2}, we obtain
\begin{equation*}
    |(u)_{\mathcal{B}^{i}_{\ell}}|^{p} \leq c_{1} 2^{k- sp+ \alpha+ \beta} |(u)_{\mathcal{B}^{j}_{\ell+1}}|^{p} + C  2^{\ell \left(sp-d \right)} [u]^{p}_{W^{s, p}(\mathcal{B}^{i}_{\ell} \cup \mathcal{B}^{j}_{\ell+1})}  .
\end{equation*}
Multiplying the above inequality by ~$2^{\ell (d-sp+ \alpha+ \beta)}$, we get
\begin{equation*}
   2^{\ell (d-sp+ \alpha+ \beta)} |(u)_{\mathcal{B}^{i}_{\ell}}|^{p} \leq c_{1} 2^{(\ell+1)(d-sp+ \alpha+ \beta) + (k-d)}|(u)_{\mathcal{B}^{j}_{\ell+1}}|^{p} + C 2^{\ell (\alpha+ \beta)} [u]^{p}_{W^{s, p}(\mathcal{B}^{i}_{\ell} \cup \mathcal{B}^{j}_{\ell+1})}  .
\end{equation*}
Since, there are ~$2^{d-k}$ such ~$\mathcal{B}^{i}_{\ell}$'s that satisfies ~\eqref{codn on Al and Al+1}. Therefore, summing the above inequality from ~$i=2^{d-k}(j-1)+1$ to $2^{d-k}j$ and using the fact $2^{\ell (\alpha+ \beta)} \leq C |x_{k}|^{\alpha}|y_{k}|^{\beta}$ for some positive constant $C$ depends on $\alpha$ and $\beta$, and for all $(x_{k}, x_{d-k}), (y_{k}, y_{d-k}) \in \mathcal{B}^{i}_{\ell} \cup \mathcal{B}^{j}_{\ell+1}$, we arrive
\begin{eqnarray*}
     2^{\ell (d-sp+ \alpha+ \beta)}  \sum_{i=2^{d-k}(j-1)+1}^{2^{d-k}j} |(u)_{\mathcal{B}^{i}_{\ell}}|^{p} &\leq& c_{1}  2^{(\ell+1)(d-sp+\alpha+ \beta)} |(u)_{\mathcal{B}^{j}_{\ell+1}}|^{p} \\ && \hspace{3mm}
       + \ C  \sum_{i=2^{d-k}(j-1)+1}^{2^{d-k}j} [u]^{p}_{W^{s, p, \alpha+ \beta}(\mathcal{B}^{i}_{\ell} \cup \mathcal{B}^{j}_{\ell+1})}  .
\end{eqnarray*}
Again, summing the above inequality from ~$j=1$ to ~$\sigma_{\ell+1}$, using the fact that
\begin{equation}\label{summation relation}
  \sum_{j=1}^{\sigma_{\ell+1}} \Bigg(  \sum_{i=2^{d-k}(j-1)+1}^{2^{d-k}j}  |(u)_{\mathcal{B}^{i}_{\ell}}|^{p} \Bigg)  = \sum_{i=1}^{\sigma_{\ell}}  |(u)_{\mathcal{B}^{i}_{\ell}}|^{p},
\end{equation}
we obtain
\begin{eqnarray*}\label{ineqn100}
2^{\ell (d-sp+ \alpha+ \beta)}  \sum_{i=1}^{\sigma_{\ell}} |(u)_{\mathcal{B}^{i}_{\ell}}|^{p} &\leq& c_{1} 2^{(\ell+1)(d-sp+ \alpha+ \beta)} \sum_{j=1}^{\sigma_{\ell+1}} |(u)_{\mathcal{B}^{j}_{\ell+1}}|^{p} \\ && \hspace{3mm}   + \ C \sum_{j=1}^{\sigma_{\ell+1}} \Bigg( \sum_{i=2^{d-k}(j-1)+1}^{2^{d-k}j}  [u]^{p}_{W^{s, p, \alpha, \beta}(\mathcal{B}^{i}_{\ell} \cup \mathcal{B}^{j}_{\ell+1})} \Bigg) \\ &\leq& c_{1} 2^{(\ell+1)(d-sp+ \alpha+ \beta)} \sum_{j=1}^{\sigma_{\ell+1}} |(u)_{\mathcal{B}^{j}_{\ell+1}}|^{p} + C  [u]^{p}_{W^{s, p, \alpha, \beta}(\mathcal{B}_{\ell} \cup \mathcal{B}_{\ell+1})} .
\end{eqnarray*}
For simplicity let ~$\mathcal{A}_{\ell} = \sum_{i=1}^{\sigma_{\ell}} |(u)_{\mathcal{B}^{i}_{\ell}}|^{p}$. Then, the above inequality will become
\begin{equation*}
    2^{\ell (d-sp+ \alpha+ \beta)} \mathcal{A}_{\ell} \leq c_{1} 2^{(\ell+1)(d-sp+ \alpha+ \beta)} \mathcal{A}_{\ell+1} + C  [u]^{p}_{W^{s, p, \alpha, \beta}(\mathcal{B}_{\ell} \cup \mathcal{B}_{\ell+1})}  .
\end{equation*}
Summing the above inequality from ~$\ell=m \in \mathbb{Z}^{-}$ to ~$n_{0}$, we get
\begin{equation*}
   \sum_{\ell=m}^{n_{0}} 2^{\ell (d-sp+ \alpha+ \beta)} \mathcal{A}_{\ell} \leq \sum_{\ell=m}^{n_{0}} c_{1} 2^{(\ell+1)(d-sp+\alpha+\beta)} \mathcal{A}_{\ell+1} + C  \sum_{\ell=m}^{n_{0}} [u]^{p}_{W^{s, p, \alpha, \beta}(\mathcal{B}_{\ell} \cup \mathcal{B}_{\ell+1})}  .
\end{equation*}
By changing sides, rearranging, and re-indexing, we get
\begin{equation*}
   2^{m(d-sp+ \alpha+ \beta)} \mathcal{A}_{m} + (1-c_{1}) \sum_{\ell=m+1}^{n_{0}} 2^{\ell (d-sp+ \alpha+ \beta)} \mathcal{A}_{\ell}  \leq  C \mathcal{A}_{n_{0}+1} 
    +  C \sum_{\ell=m}^{n_{0}} [u]^{p}_{W^{s, p, \alpha, \beta}(\mathcal{B}_{\ell} \cup \mathcal{B}_{\ell+1})}.
\end{equation*}
Choose ~$-m$ sufficiently large so that ~$|(u)_{\mathcal{B}^{j}_{m}}|=0$ for all ~$j \in \{ 1, \dots, \sigma_{m} \}$. Since $supp(u) \subset \mathcal{D}$, this implies that $|(u)_{\mathcal{B}^{j}_{n_{0}+1}}|=0$ for all ~$j \in \{ 1, \dots, \sigma_{n_{0}+1} \}$. Hence, we have for some constant ~$C= C(d,p,s,k, \alpha, \beta)>0$,
\begin{equation*}
    (1-c_{1}) \sum_{\ell=m}^{n_{0}} 2^{\ell (d-sp+ \alpha+ \beta)} \mathcal{A}_{\ell} \leq    C \sum_{\ell=m}^{n_{0}} [u]^{p}_{W^{s, p, \alpha, \beta}(\mathcal{B}_{\ell} \cup \mathcal{B}_{\ell+1})}  .
\end{equation*}
Putting the value of ~$\mathcal{A}_{\ell}$ in the above inequality, we obtain
\begin{equation}\label{eqnn99}
 (1-c_{1})   \sum_{\ell=m}^{n_{0}} 2^{\ell (d-sp+ \alpha+ \beta)} \sum_{i=1}^{\sigma_{\ell}} |(u)_{\mathcal{B}^{i}_{\ell}}|^{p} \leq  C \sum_{\ell=m}^{n_{0}} [u]^{p}_{W^{s, p, \alpha, \beta}(\mathcal{B}_{\ell} \cup \mathcal{B}_{\ell+1})}  .
\end{equation}
Combining ~\eqref{eqnn1} and ~\eqref{eqnn99} yields
\begin{equation*}
    \sum_{\ell=m}^{n_{0}} \int_{\mathcal{B}_{\ell}} \frac{|u(x)|^{p}}{|x_{k}|^{sp- \alpha- \beta} }  dx \leq  C \sum_{\ell=m}^{n_{0}} [u]^{p}_{W^{s, p, \alpha, \beta}(\mathcal{B}_{\ell} \cup \mathcal{B}_{\ell+1})}  .
\end{equation*}
Therefore, from Lemma \ref{lemma 3}, we have
 \begin{equation*}
       \bigintssss_{\mathbb{R}^{d}} \frac{|u(x)|^{p}}{|x_{k}|^{sp- \alpha- \beta} }  dx  \leq C  [u]^{p}_{W^{s, p, \alpha, \beta}(\mathbb{R}^{d})} .
 \end{equation*}
This finishes the proof of Theorem \ref{theorem1} when $sp<k+\alpha+\beta$.

\subsection{Proof of Theorem \ref{theorem1} when \texorpdfstring{$sp>k+\alpha+\beta$}{sp>k}}
Let ~$u \in C^{1}_{c}(\mathbb{R}^{d} \backslash K)$ such that ~$supp(u) \subset B^{k}(0, 2^{n_{0}+1}) \times (-2^{n_{1}},2^{n_{1}})^{d-k} = \mathcal{D}$ for some ~$n_{0}, n_{1} \in \mathbb{Z}$ satisfying ~$2^{n_{1}} \geq 2^{n_{0}+1}$. Let ~$\mathcal{B}^{j}_{\ell+1}$ be such that it satisfies ~\eqref{codn on Al and Al+1}. Using triangle inequality, we have 
\begin{equation*}
    |(u)_{\mathcal{B}^{j}_{\ell+1}}|^p \leq  \left( |(u)_{\mathcal{B}^{i}_{\ell}}| + |(u)_{\mathcal{B}^{i}_{\ell}} - (u)_{\mathcal{B}^{j}_{\ell+1}}| \right)^p.
\end{equation*}
We know that $sp>k+\alpha+ \beta$, applying Lemma ~\ref{estimate} with ~$c :=c_{1}2^{sp-\alpha-\beta-k}>1$ where ~$c_{1} = \frac{2}{1+ 2^{sp-\alpha-\beta-k}} <1$ together with Lemma ~\ref{est2}, we obtain
\begin{equation*}
    |(u)_{\mathcal{B}^{j}_{\ell+1}}|^{p} \leq c_{1} 2^{sp-\alpha-\beta-k} |(u)_{\mathcal{B}^{i}_{\ell}}|^{p} + C  2^{\ell(sp-d) } [u]^{p}_{W^{s, p}(\mathcal{B}^{i}_{\ell} \cup \mathcal{B}^{j}_{\ell+1})}  .
\end{equation*}
Multiplying the above inequality by ~$2^{(\ell+1)(d-sp+\alpha+ \beta)-d+k}$, we get
\begin{equation*}
   2^{(\ell+1)(d-sp+\alpha+ \beta)-d+k} |(u)_{\mathcal{B}^{j}_{\ell+1}}|^{p} \leq c_{1} 2^{\ell (d-sp+\alpha+ \beta)}|(u)_{\mathcal{B}^{i}_{\ell}}|^{p} + C 2^{\ell(\alpha+ \beta)} [u]^{p}_{W^{s, p}(\mathcal{B}^{i}_{\ell} \cup \mathcal{B}^{j}_{\ell+1})}  .
\end{equation*}
Since, there are ~$2^{d-k}$ such ~$\mathcal{B}^{i}_{\ell}$'s that satisfies ~\eqref{codn on Al and Al+1}. Therefore, summing the above inequality from ~$i=2^{d-k}(j-1)+1$ to ~$2^{d-k}j$ and using the fact $2^{\ell (\alpha+ \beta)} \leq C |x_{k}|^{\alpha}|y_{k}|^{\beta}$ for some positive constant $C$ depends on $\alpha$ and $\beta$, and for all $(x_{k}, x_{d-k}), (y_{k}, y_{d-k}) \in \mathcal{B}^{i}_{\ell} \cup \mathcal{B}^{j}_{\ell+1}$, we obtain
\begin{eqnarray*}
    2^{(\ell+1)(d-sp+\alpha+ \beta)}  |(u)_{\mathcal{B}^{j}_{\ell+1}}|^{p} &\leq& c_{1} 2^{\ell (d-sp+\alpha+ \beta)} \sum_{i=2^{d-k}(j-1)+1}^{2^{d-k}j}  |(u)_{\mathcal{B}^{i}_{\ell}}|^{p} \\ && \hspace{3mm}
       +  \ C  \sum_{i=2^{d-k}(j-1)+1}^{2^{d-k}j} [u]^{p}_{W^{s, p, \alpha, \beta}(\mathcal{B}^{i}_{\ell} \cup \mathcal{B}^{j}_{\ell+1})}  .
\end{eqnarray*}
Again, summing the above inequality from ~$j=1$ to ~$\sigma_{\ell+1}$ and using ~\eqref{summation relation}, we obtain
\begin{eqnarray*}\label{ineqn1}
2^{(\ell+1)(d-sp+\alpha+ \beta)}  \sum_{j=1}^{\sigma_{\ell+1}}  |(u)_{\mathcal{B}^{j}_{\ell+1}}|^{p} &\leq& c_{1} 2^{\ell (d-sp+\alpha+ \beta)} 
\sum_{i=1}^{\sigma_{\ell}} |(u)_{\mathcal{B}^{i}_{\ell}}|^{p} \nonumber \\ && \hspace{2mm}
  + \ C  \sum_{j=1}^{\sigma_{\ell+1}} \Bigg( \sum_{i=2^{d-k}(j-1)+1}^{2^{d-k}j}  [u]^{p}_{W^{s, p, \alpha, \beta}(\mathcal{B}^{i}_{\ell} \cup \mathcal{B}^{j}_{\ell+1})} \Bigg) \nonumber \\ &\leq& c_{1} 2^{\ell (d-sp+\alpha+ \beta)} \sum_{i=1}^{\sigma_{\ell}} |(u)_{\mathcal{B}^{i}_{\ell}}|^{p} + C  [u]^{p}_{W^{s, p, \alpha, \beta}(\mathcal{B}_{\ell} \cup \mathcal{B}_{\ell+1})}   .
\end{eqnarray*}
Summing the above inequality from ~$\ell=m \in \mathbb{Z}^{-}$ to ~$n_{0}-1$, we get
\begin{eqnarray*}
   \sum_{\ell=m}^{n_{0}-1} 2^{(\ell+1)(d-sp+ \alpha+ \beta)}  \sum_{j=1}^{\sigma_{\ell+1}}  |(u)_{\mathcal{B}^{j}_{\ell+1}}|^{p} &\leq& c_{1} \sum_{\ell=m}^{n_{0}-1}  2^{\ell (d-sp+ \alpha+ \beta)} \sum_{i=1}^{\sigma_{\ell}} |(u)_{\mathcal{B}^{i}_{\ell}}|^{p} \\ && \hspace{3mm} + \  C  \sum_{\ell=m}^{n_{0}-1} [u]^{p}_{W^{s, p, \alpha, \beta}(\mathcal{B}_{\ell} \cup \mathcal{B}_{\ell+1})}  .
\end{eqnarray*}
By changing sides, rearranging, and re-indexing, we get
\begin{eqnarray*}
 (1-c_{1})   \sum_{\ell=m+1}^{n_{0}} 2^{\ell (d-sp+\alpha+ \beta)} \sum_{i=1}^{\sigma_{\ell}} |(u)_{\mathcal{B}^{i}_{\ell}}|^{p} &\leq & 2^{m(d- sp+\alpha+ \beta)} \sum_{j=1}^{\sigma_{m}} |(u)_{\mathcal{B}^{j}_{m}}|^{p} \\ && \hspace{3mm} + \ C \sum_{\ell=m}^{n_{0}-1} [u]^{p}_{W^{s, p, \alpha, \beta}(\mathcal{B}_{\ell} \cup \mathcal{B}_{\ell+1})}  .
\end{eqnarray*}
Choose ~$-m$ sufficiently large so that ~$|(u)_{\mathcal{B}^{j}_{m}}|=0$ for all ~$j \in \{ 1, \dots, \sigma_{m} \}$. Therefore, we have
\begin{equation}\label{eqnn}
    (1-c_{1})   \sum_{\ell=m}^{n_{0}} 2^{\ell (d-sp+\alpha + \beta)} \sum_{i=1}^{\sigma_{\ell}} |(u)_{\mathcal{B}^{i}_{\ell}}|^{p} \leq C \sum_{\ell=m}^{n_{0}-1} [u]^{p}_{W^{s, p, \alpha, \beta}(\mathcal{B}_{\ell} \cup \mathcal{B}_{\ell+1})}  .
\end{equation}
Combining ~\eqref{eqnn1} and ~\eqref{eqnn} yields
\begin{equation*}
    \sum_{\ell=m}^{n_{0}} \int_{\mathcal{B}_{\ell}} \frac{|u(x)|^{p}}{|x_{k}|^{sp-\alpha- \beta} }  dx \leq   C \sum_{\ell=m}^{n_{0}} [u]^{p}_{W^{s, p, \alpha, \beta}(\mathcal{B}_{\ell} \cup \mathcal{B}_{\ell+1})}  .
\end{equation*}
Therefore, from Lemma \ref{lemma 3}, we have
 \begin{equation*}
      \int_{\mathbb{R}^{d}} \frac{|u(x)|^{p}}{|x_{k}|^{sp- \alpha- \beta} }  dx  \leq C [u]^{p}_{W^{s, p, \alpha, \beta}(\mathbb{R}^{d})}.
 \end{equation*}
This finishes the proof of Theorem \ref{theorem1} when $sp>k+ \alpha+ \beta$.

\bigskip

In the following subsection, we will establish Theorem ~\ref{theorem2}, which represents the critical case of Theorem \ref{theorem1}, i.e.,  for the case $sp=k+\alpha+\beta$.  

 \subsection{Proof of Theorem \ref{theorem2}} 
 In this case $sp= k + \alpha+ \beta$. Let ~$u \in C^{1}_{c}(\mathbb{R}^{d} \backslash K)$ such that $supp(u) \subset B^{k}(0, R) \times \mathbb{R}^{d-k}$ for some $R>0$. Choose $n_{0}, ~ n_{1} \in \mathbb{Z}$ such that $2^{n_{0}} \leq R < 2^{n_{0}+1}$ and $2^{n_{1}} \geq 2^{n_{0}+1}$. Therefore, $supp(u) \subset B^{k}(0,2^{n_{0}+1}) \times (-2^{n_{1}}, 2^{n_{1}})^{d-k} = \mathcal{D}$. Since $2^{n_{0}}\leq R < 2^{n_{0}+1}$ and for each ~$x= (x_{k}, x_{d-k}) \in \mathcal{B}^{i}_{\ell}$, we have ~$|x_{k}| < 2^{\ell +1}$ which  implies ~$ \ln \left(\frac{4R}{|x_{k}|} \right) > (n_{0}-\ell+1) \ln 2$. Using this together with Lemma ~\ref{sumineqlemma} and ~$\frac{1}{(n_{0}-\ell+1)^{p}} \leq 1$, we obtain
\begin{eqnarray*}
    \bigintssss_{\mathcal{B}^{i}_{\ell}} \frac{|u(x)|^{p}}{|x_{k}|^{k} \ln^{p} \left(\frac{4R}{|x_{k}|} \right)}  dx &\leq& \frac{C}{(n_{0}-\ell+1)^{p}} [u]^{p}_{W^{s, p, \alpha, \beta}(\mathcal{B}^{i}_{\ell})} + C \frac{2^{\ell(d-k)}}{(n_{0}-\ell+1)^{p}} |(u)_{\mathcal{B}^{i}_{\ell}}|^{p} \\ 
     &\leq&  C [u]^{p}_{W^{s, p, \alpha, \beta}(\mathcal{B}^{i}_{\ell})} + C \frac{2^{\ell(d-k)}}{(n_{0}-\ell+1)^{p}} |(u)_{\mathcal{B}^{i}_{\ell}}|^{p}  .
\end{eqnarray*}
Summing the above inequality from ~$i=1$ to ~$\sigma_{\ell}$, we get
\begin{eqnarray*}
    \bigintssss_{\mathcal{B}_{\ell}} \frac{|u(x)|^{p}}{|x_{k}|^{k} \ln^{p} \left( \frac{4R}{|x_{k}|} \right)}  dx &\leq& C   \sum_{i=1}^{\sigma_{\ell}} [u]^{p}_{W^{s, p, \alpha, \beta}(\mathcal{B}^{i}_{\ell})} + C \frac{2^{\ell (d-k)}}{(n_{0}-\ell+1)^{p}}  \sum_{i=1}^{\sigma_{\ell}} |(u)_{\mathcal{B}^{i}_{\ell}}|^{p} \\ &\leq& C [u]^{p}_{W^{s, p, \alpha, \beta}(\mathcal{B}_{\ell})} + C \frac{2^{\ell (d-k)}}{(n_{0}-\ell+1)^{p}}  \sum_{i=1}^{\sigma_{\ell}} |(u)_{\mathcal{B}^{i}_{\ell}}|^{p}.
\end{eqnarray*}
Again, summing the above inequality from ~$\ell =m \in \mathbb{Z}^{-}$ to ~$n_{0}$, we get
\begin{equation}\label{eqnn2}
\sum_{\ell=m}^{n_{0}} \bigintssss_{\mathcal{B}_{\ell}} \frac{|u(x)|^{p}}{|x_{k}|^{k} \ln^{p}\left( \frac{4R}{|x_{k}|} \right)}  dx \leq  C \sum_{\ell=m}^{n_{0}} [u]^{p}_{W^{s, p, \alpha, \beta}(\mathcal{B}_{\ell})} + C \sum_{\ell=m}^{n_{0}} \frac{2^{\ell (d-k)}}{(n_{0}-\ell+1)^{p}}  \sum_{i=1}^{\sigma_{\ell}} |(u)_{\mathcal{B}^{i}_{\ell}}|^{p}  .
\end{equation}
Let ~$\mathcal{B}^{j}_{\ell+1}$ be such that it satisfies ~\eqref{codn on Al and Al+1}. Using triangle inequality, we have 
\begin{equation*}
    |(u)_{\mathcal{B}^{i}_{\ell}}|^p \leq  \left( |(u)_{\mathcal{B}^{j}_{\ell+1}}| + |(u)_{\mathcal{B}^{i}_{\ell}} - (u)_{\mathcal{B}^{j}_{\ell+1}}| \right)^p.
\end{equation*}
For ~$\ell \leq n_{0}$, applying Lemma ~\ref{estimate} with ~$c := \left( \frac{n_{0}-\ell+1}{n_{0}-\ell+(1/2)} \right)^{p-1}>1$ together with Lemma ~\ref{est2} and ~$sp=k+ \alpha+ \beta$ and using the fact using the fact $2^{\ell (\alpha+ \beta)} \leq C |x_{k}|^{\alpha}|y_{k}|^{\beta}$ for some positive constant $C$ depends on $\alpha$ and $\beta$, and for all $(x_{k}, x_{d-k}), (y_{k}, y_{d-k}) \in \mathcal{B}^{i}_{\ell} \cup \mathcal{B}^{j}_{\ell+1}$, we obtain
\begin{eqnarray*}
    |(u)_{\mathcal{B}^{i}_{\ell}}|^{p} &\leq& \left( \frac{n_{0}-\ell+1}{n_{0}-\ell+(1/2)} \right)^{p-1} |(u)_{\mathcal{B}^{j}_{\ell+1}}|^{p} + C (n_{0}-\ell+1)^{p-1} 2^{\ell(k+ \alpha+ \beta-d) } [u]^{p}_{W^{s, p}(\mathcal{B}^{i}_{\ell} \cup \mathcal{B}^{j}_{\ell+1})}  \\ &\leq& \left( \frac{n_{0}-\ell+1}{n_{0}-\ell+(1/2)} \right)^{p-1} |(u)_{\mathcal{B}^{j}_{\ell+1}}|^{p} + C (n_{0}-\ell+1)^{p-1} 2^{\ell(k-d) } [u]^{p}_{W^{s, p, \alpha, \beta}(\mathcal{B}^{i}_{\ell} \cup \mathcal{B}^{j}_{\ell+1})} .
\end{eqnarray*}
Multiplying the above inequality by ~$2^{\ell (d-k)}$, we get
\begin{equation*}
    \frac{2^{\ell (d-k)}}{(n_{0}-\ell+1)^{p-1}}  |(u)_{\mathcal{B}^{i}_{\ell}}|^{p} \leq \frac{2^{\ell (d-k)}}{ (n_{0}-\ell+(1/2))^{p-1}} |(u)_{\mathcal{B}^{j}_{\ell+1}}|^{p} + C [u]^{p}_{W^{s, p, \alpha, \beta}(\mathcal{B}^{i}_{\ell} \cup \mathcal{B}^{j}_{\ell+1})}  .
\end{equation*}
Since, there are ~$2^{d-k}$ such ~$\mathcal{B}^{i}_{\ell}$'s that satisfies ~\eqref{codn on Al and Al+1}. Therefore, summing the above inequality from ~$i=2^{d-k}(j-1)+1$ to ~$2^{d-k}j$, we obtain
\begin{eqnarray*}
      \frac{2^{\ell (d-k)}}{(n_{0}-\ell+1)^{p-1}}  \sum_{i=2^{d-k}(j-1)+1}^{2^{d-k}j} |(u)_{\mathcal{B}^{i}_{\ell}}|^{p} &\leq&  \frac{2^{(\ell+1) (d-k)}}{ (n_{0}-\ell+(1/2))^{p-1}} |(u)_{\mathcal{B}^{j}_{\ell+1}}|^{p} \\ && \hspace{3mm}
       + \ C  \sum_{i=2^{d-k}(j-1)+1}^{2^{d-k}j} [u]^{p}_{W^{s, p, \alpha, \beta}(\mathcal{B}^{i}_{\ell} \cup \mathcal{B}^{j}_{\ell+1})}  .
\end{eqnarray*}
Again, summing the above inequality from ~$j=1$ to ~$\sigma_{\ell+1}$ and using ~\eqref{summation relation}, we obtain
\begin{eqnarray*}
\frac{2^{\ell(d-k)}}{(n_{0}-\ell+1)^{p-1}}  \sum_{i=1}^{\sigma_{\ell}} |(u)_{\mathcal{B}^{i}_{\ell}}|^{p} &\leq& \frac{2^{(\ell+1)(d-k)}}{ (n_{0}-\ell+(1/2))^{p-1}} \sum_{j=1}^{\sigma_{\ell+1}} |(u)_{\mathcal{B}^{j}_{\ell+1}}|^{p} \\ && \hspace{3mm}
  + \ C \sum_{j=1}^{\sigma_{\ell+1}} \Bigg( \sum_{i=2^{d-k}(j-1)+1}^{2^{d-k}j}  [u]^{p}_{W^{s, p, \alpha, \beta}(\mathcal{B}^{i}_{\ell} \cup \mathcal{B}^{j}_{\ell+1})} \Bigg) \\ &\leq& \frac{2^{(\ell+1)(d-k)}}{ (n_{0}-\ell+(1/2))^{p-1}} \sum_{j=1}^{\sigma_{\ell+1}} |(u)_{\mathcal{B}^{j}_{\ell+1}}|^{p} + C  [u]^{p}_{W^{s, p, \alpha, \beta}(\mathcal{B}_{\ell} \cup \mathcal{B}_{\ell+1})} .
\end{eqnarray*}
For simplicity let ~$\mathcal{A}_{\ell} = \sum_{i=1}^{\sigma_{\ell}} |(u)_{\mathcal{B}^{i}_{\ell}}|^{p}$. Then, the above inequality will become
\begin{equation*}
    \frac{2^{\ell(d-k)}}{(n_{0}-\ell+1)^{p-1}} \mathcal{A}_{\ell} \leq \frac{2^{(\ell+1)(d-k)}}{ (n_{0}-\ell+(1/2))^{p-1}} \mathcal{A}_{\ell+1} + C  [u]^{p}_{W^{s, p, \alpha, \beta}(\mathcal{B}_{\ell} \cup \mathcal{B}_{\ell+1})}  .
\end{equation*}
Summing the above inequality from ~$\ell=m \in \mathbb{Z}^{-}$ to ~$n_{0}$, we get
\begin{equation*}
   \sum_{\ell=m}^{n_{0}} \frac{2^{\ell(d-k)}}{(n_{0}-\ell+1)^{p-1}} \mathcal{A}_{\ell} \leq \sum_{\ell=m}^{n_{0}} \frac{2^{(\ell+1)(d-k)}}{ (n_{0}-\ell+(1/2))^{p-1}} \mathcal{A}_{\ell+1} + C  \sum_{\ell=m}^{n_{0}} [u]^{p}_{W^{s, p, \alpha, \beta}(\mathcal{B}_{\ell} \cup \mathcal{B}_{\ell+1})}  .
\end{equation*}
By changing sides, rearranging, and re-indexing, we get
\begin{multline*}
    \frac{2^{m(d-k)}}{(n_{0}-m+1)^{p-1}} \mathcal{A}_{m} + \sum_{\ell=m+1}^{n_{0}} \left\{ \frac{1}{(n_{0}-\ell+1)^{p-1}} - \frac{1}{(n_{0}-\ell+3/2)^{p-1}} \right\} 2^{\ell(d-k)} \mathcal{A}_{\ell}  \\ \leq 2^{p-1} 2^{(n_{0}+1)(d-k)}  \mathcal{A}_{n_{0}+1} 
    + \ C \sum_{\ell=m}^{n_{0}} [u]^{p}_{W^{s, p, \alpha, \beta}(\mathcal{B}_{\ell} \cup \mathcal{B}_{\ell+1})}  .
\end{multline*}
Now, for large values of ~$-\ell$, we have
\begin{equation*}
     \frac{1}{(n_{0}-\ell+1)^{p-1}} - \frac{1}{(n_{0}-\ell+3/2)^{p-1}} \sim \frac{1}{(n_{0}-\ell+1)^{p}},
\end{equation*}
and choose ~$-m$ sufficiently large so that ~$|(u)_{\mathcal{B}^{j}_{m}}|=0$ for all ~$j \in \{ 1, \dots, \sigma_{m} \}$. Since $supp(u) \subset \mathcal{D}$, this implies that $|(u)_{\mathcal{B}^{j}_{n_{0}+1}}|=0$ for all ~$j \in \{ 1, \dots, \sigma_{n_{0}+1} \}$. Hence, we have for some constant ~$C= C(d,p,k, \alpha, \beta)>0$,
\begin{equation*}
    \sum_{\ell=m}^{n_{0}} \frac{2^{\ell(d-k)}}{(n_{0}-\ell+1)^{p}} \mathcal{A}_{\ell} \leq    C \sum_{\ell=m}^{n_{0}} [u]^{p}_{W^{s, p, \alpha, \beta}(\mathcal{B}_{\ell} \cup \mathcal{B}_{\ell+1})}  .
\end{equation*}
Putting the value of ~$\mathcal{A}_{\ell}$ in the above inequality, we obtain
\begin{equation}\label{eqn10}
    \sum_{\ell=m}^{n_{0}} \frac{2^{\ell(d-k)}}{(n_{0}-\ell+1)^{p}} \sum_{i=1}^{\sigma_{\ell}} |(u)_{\mathcal{B}^{i}_{\ell}}|^{p} \leq  C \sum_{\ell=m}^{n_{0}} [u]^{p}_{W^{s, p, \alpha, \beta}(\mathcal{B}_{\ell} \cup \mathcal{B}_{\ell+1})}  .
\end{equation}
Combining ~\eqref{eqnn2} and ~\eqref{eqn10} yields
\begin{equation*}\label{eqqnn}
    \sum_{\ell=m}^{n_{0}} \bigintssss_{\mathcal{B}_{\ell}} \frac{|u(x)|^{p}}{|x_{k}|^{k} \ln^{p} \left( \frac{4R}{|x_{k}|} \right)}  dx \leq  C \sum_{\ell=m}^{n_{0}} [u]^{p}_{W^{s, p, \alpha, \beta}(\mathcal{B}_{\ell} \cup \mathcal{B}_{\ell+1})}  .
\end{equation*}
Therefore, from Lemma \ref{lemma 3}, we have
 \begin{equation*}
       \bigintssss_{\mathbb{R}^{d}} \frac{|u(x)|^{p}}{|x_{k}|^{k} \ln^{p}\left( \frac{4R}{|x_{k}|} \right)}  dx \leq C [u]^{p}_{W^{s, p, \alpha, \beta}(\mathbb{R}^{d})}.
 \end{equation*}
This finishes the proof of Theorem \ref{theorem2}.

\bigskip

\noindent \textbf{Acknowledgement:} We extend our sincere thanks to the Department of Mathematics and Statistics at the Indian Institute of Technology Kanpur, India, for fostering a supportive research environment. V. Sahu gratefully acknowledges the financial support provided by the MHRD, Government of India, through the GATE fellowship. The author would like to thank Prof. Prosenjit Roy for useful discussion on the subject.

\end{document}